\newcommand{\les}{\lesssim}

\documentstyle[amssymb,amsfonts]{amsart}

\newenvironment{proof}{\noindent {\bf Proof} }{\endprf\par}
\def \endprf{\hfill  {\vrule height6pt width6pt depth0pt}\medskip}
\def\emph#1{{\it #1}}
\def\textbf#1{{\bf #1}}
\newcommand{\bea}{\begin{eqnarray}}
\newcommand{\eea}{\end{eqnarray}}
\def\beaa{\begin{eqnarray*}}
\def\eeaa{\end{eqnarray*}}

\def\ba{\begin{array}}
\def\ea{\end{array}}
\def\be#1{\begin{equation} \label{#1}}
\def \eeq{\end{equation}}

\newcommand{\nn}{\nonumber}

\def\NNN{{\Bbb N}}

\def\a{{\alpha}}
\def\b{{\beta}}

\def\De{\Delta}

\def\la{\lambda}

\def\nab{\nabla}

\def\trch{\mbox{tr}\chi}

\def\NN{{\cal N}}

\def\HH{{\cal H}}

\def\RR{{\cal R}}

\def\HH{{\cal H}}

\def\k{{\bold k}}

\def\RRR{{\Bbb R}}
\def\ZZZ{{\Bbb Z}}

\def\lap{\Delta}
\def\pr{\partial}

\def\c{\cdot}

\def\f14{{\frac{1}{4}}}
\def\f12{{\frac{1}{2}}}

\def\2{{\overline 2}}

\begin{document}
\theoremstyle{plain}
  \newtheorem{theorem}[subsection]{Theorem}
  \newtheorem{conjecture}[subsection]{Conjecture}
  \newtheorem{proposition}[subsection]{Proposition}
  \newtheorem{lemma}[subsection]{Lemma}
  \newtheorem{corollary}[subsection]{Corollary}

\theoremstyle{remark}
  \newtheorem{remark}[subsection]{Remark}
  \newtheorem{remarks}[subsection]{Remarks}

\theoremstyle{definition}
  \newtheorem{definition}[subsection]{Definition}

\include{psfig}

\include{psfig}
\title[Transport equations ]{Sharp $L^1$ estimates for singular
transport equations }
\author{Sergiu Klainerman}
\address{Department of Mathematics, Princeton University,
 Princeton NJ 08544}
\email{ seri@@math.princeton.edu}

\author{Igor Rodnianski}
\address{Department of Mathematics, Princeton University, 
Princeton NJ 08544}
\email{ irod@@math.princeton.edu}
\subjclass{35J10\newline\newline
The first author is partially supported by NSF grant 
DMS-0070696. The second author is partially 
supported by NSF grant DMS-0406627
}
\begin{abstract}
We provide $L^1$ estimates for 
a class of transport equations containing singular
integral operators. The form of the equation was motivated by the study of 
Kirchhoff-Sobolev parametrices   in a Lorentzian space-time
verifying  the Einstein equations. While
our main application is for a specific problem in General Relativity we believe that the 
phenomenon which our result illustrates
is of a more  general interest.

\end{abstract}
\maketitle
\vspace{-0.3in}
\section{Introduction}
The goal of this paper is to prove an $L^1$ type estimate
for solutions of the following transport equation,
\bea
\label{eq:Tr1}
\pr_t u(t,x)- a(t,x) M u(t,x)=g(t,x), \qquad u(0,x)=0.
\eea
Here $a=a(t,x)$ and $g=g(t,x)$ are  assumed to be smooth, compactly
  supported functions defined\footnote{Similar results can be
easily extended to higher dimensions.} on
$[0,1]\times\RRR^2$ and
$M$ is  a classical, translation invariant,  Calderon-Zygmund
operator in $\RRR^2$, given by a smooth\footnote{The smoothness assumption is only   imposed to 
eliminate logarithmic divergences at infinity in $\RRR^2$, is
 irrelevant to our  main concerns.} multiplier. 
Though, for simplicity,  we  shall proceed as
if the equation 
\eqref{eq:Tr1} is scalar, all our results extend easily to
systems, i.e. $u$ and $g$ take values in $\RRR^N$ and $aM$
is a $N\times N$ matrix valued operator.

Ideally, the desired estimate would take the form 
$$
\sup_{t\in [0,1]}\|u(t)\|_{L^1(\RRR^2)} \le C(\|a\|_{L^\infty([0,1]\times \RRR^2)})\,
\|g\|_{L^1([0,1]\times\RRR^2)}
$$
As it is well known however such 
 $L^1$-type estimates cannot possibly hold 
 due to 
  the failure of  $L^1$ boundedness of  Calderon-Zygmund operators.
To illustrate  this consider first the case of a constant coefficient transport equation with $a\equiv 1$. 
In this case we may write
\bea
u(t,x)=\int_0^te^{(t-s)M} g(s) ds
\eea
where,
\beaa
e^{tM}=I+tM+\frac{1}{2} (tM)^2+\cdots +\frac{1}{n!}(tM)^n+\cdots
\eeaa
The problem of $L^1$ estimates for \eqref{eq:Tr1} is then reduced to the corresponding 
question for the operators $M^n$. Each of $M^n$ is a Calderon-Zygmund operator and 
as such  does not map $L^1$ to $L^1$.   A well known    way
to resolve   this problem    is to consider instead  mapping   properties
of the Hardy space\footnote{The classical Hardy space $\HH_1$,
 defined by the norm  $
\|f\|_{{\cal H}_1}=\|f\|_{L^1(\RRR^2)} + \sup_{j=1,2} \|R_j f\|_{L^1(\RRR^2)}$, is a 
can be viewed as a logarithmic improvement of $L^1$. 
Here   $R_j=(-\lap)^{1/2} \pr_j$  are  the standard Riesz operators 
in $\RR^2$.}  $\HH_1$ to $L^1$.  Since
 translation invariant  Calderon-Zygmund operators  $M$ map ${\cal H}_1$ into ${\cal H}_1$ 
(see \cite{Stein}) we easily infer 
that a solution $u$ of the transport equation 
$$
\pr_t u-Mu=g,\qquad u(0,x)=0
$$
belongs to the space $L^\infty([0,1];{\cal H}_1)$. Indeed,
\beaa
\|u(t)\|_{{\cal H}_1}& \le& \sum_{n=0}^\infty \int_0^t \frac{(t-s)^n}{n!} \|M^n g(s)\|_{{\cal H}_1}
\le \sum_{n=0}^\infty \int_0^t \frac{C^n (t-s)^n}{n!} \|g(s)\|_{{\cal H}_1}\, ds \\ &\le& e^{Ct}
\int_0^t \|g(s)\|_{{\cal H}_1} \, ds
\eeaa
While this may be considered a satisfactory solution of the problem for the transport equation
\eqref{eq:Tr1} with constant coefficients, the situation changes drastically in the variable
coefficient case. Consider the transport equation
\be{eq:Tran2}
\pr_t u-a(x) Mu=g,\qquad u(0,x)=0
\end{equation}
with a time-independent coefficient $a(x)$. As before we may write
\be{eq:expo}
u(t,x)=\int_0^te^{(t-s)aM} g(s) ds
\end{equation}
where,
\beaa
e^{t\,aM}=I+t\,aM+\frac{1}{2} (t\, aM)^2+\cdots +\frac{1}{n!}(t\, aM)^n+\cdots
\eeaa
The multiplication operator $a$ and Calderon-Zygmund operator $M$ do not 
commute\footnote{If they did
we could  write  $(a M)^n$ as $a^n M^n$ and
derive
$
 \|u(t)\|_{L^1(\RRR^2)}\le C \int_0^t \frac {\|a\|^n_{L^\infty(\RRR^2)} (t-s)^n}{n!}
\|M^n g(s) \|_{L^1(\RRR^2)}\,ds \le e^{Ct \|a\|_{L^\infty(\RRR^2)}} 
\int_0^t \|g(s)\|_{{\cal H}_1}\, ds.
$
}. 
We  need instead that  the operator
$a M$  has the same mapping properties as
$M$, i.e. it maps $\HH_1$ to itself,  in which  case 
we would easily conclude that 
  solutions of the
transport equation \eqref{eq:Tran2}  belong to the space 
$L^\infty([0,1]; {\cal H}^1)$.  To insure this condition we are led
 to the requirement that multiplication
by the function $a=a(x)$ maps Hardy space into itself. It is well known however that a multiplication
by a bounded function does not preserve ${\cal H}_1$. Instead, such a function $a$ should 
satisfy the Dini condition 
$$
\int_0^\infty \sup_{|x-y|\le \lambda} |a(x)-a(y)|\, \frac{d\lambda}{\lambda} <\infty,
$$
see \cite{Stegenga}.
Functions satisfying the Dini condition can not be sharply characterized in terms of the 
standard Lebesgue type spaces. Specifically, one can easily see that  even if $a$ is  a single
atom in the Besov space $B^0_{\infty,1}(\RRR^2)$ or even
 in $B^1_{2,1}(\RRR^2)$, both sharp Besov refinements
of the $L^\infty(\RRR^2)$ space, does not guarantee that the Dini condition  is satisfied.  Yet, in view of the specific applications
we have in mind, we need to  consider  precisely the situation
 when $a$ belongs to the space  $B^1_{2,1}$,  and allow  
even more general functions in the time-dependent case. 
As a consequence to accomplish our goal we need to give up on the
Hardy space $\HH_1$ and consider in fact  estimates\footnote{To prove such estimates we need the 
 the  symbol $m(\xi)$ of $M$ is smooth at the origin, i.e.,
$
|\pr^\a m(\xi)|\le c (1+|\xi)^{-|\alpha|},\qquad \forall \xi\in \RRR^2. 
$}  for solutions
$u$ of transport equation   \eqref{eq:Tran2}  of the form,
\be{eq:Est1}
\sup_{t\in [0,1]} \|u(t)\|_{L^1(\RRR^2)} \le C(\|a\|_{B^1_{2,1}(\RRR^2)}) N(g),
\end{equation}
where the expression $N(g)$ reflects a logarithmic loss\footnote{Recall that according to the result 
of Stein \cite{Stein1} the Hardy space ${\cal H}_1$ contains precisely such logarithmic loss, as
the finiteness of the local, i.e. the norm $\|f\|_{L^1}+\|R_j f\|_{L^1}$ computed over balls $B$, 
${\cal H}_1$ norm of $g$ is equivalent to  
bounds on $\int_B |f(x)|\log^+f(x)\, dx$.
 } relative to the $L^1$ norm of
$g$.  The proper definition of $N(g)$ is given below in \eqref{eq:AA}.
  In the  particular case of $g$ with compact support $N(g)$ 
  becomes simply
$
 \|g\|_{L^1(\RRR^2)} \log^+\|g\|_{L^\infty(\RRR^2)} +1.
$

The key feature of  estimate  \eqref{eq:Est1} is  that
only one  logarithmic loss is present.   This means
that  we are not able to attack the problem by merely considering
the mapping properties of the operator $aM$. Indeed
the best we can prove  is the estimate,
\beaa
\sup_{t\in [0,1]} \|a Mg(t)\|_{L^1(\RRR^2)} \le C(\|a\|_{B^1_{2,1}(\RRR^2)}) N(g),
\eeaa
which  leads, by iteration, to a loss of $\big(  \log^+\|g\|_{L^\infty(\RRR^2)})^n$ for $(aM)^n$.
Instead we analyze  directly the mapping properties of the 
multilinear expressions
\be{eq:multi}
(a(x) M)^n= a(x)\, M\, a(x)\, M\,....\, a(x) M
\end{equation}
and their sums.  Using 
 commutator estimates
and  appropriate interpolations between the 
weak $L^1$ and $L^2$ mapping properties of 
the operators $M$ we are able to show that
in fact we lose only one logarithm for $\|(aM)^ng\|_{L^1}$, 
 regardless of the exponent $n$.
 Note however that under our assumptions on $a(x)$ the commutator
$[a(x),M]$ is not a bounded operator\footnote{The classical result of Coifman-Rochberg-Weiss
\cite{CRW}
requires only that $a\in {\text{BMO}}$ for the commutator to be bounded on $L^p$ with $p\in (1,\infty)$. 
Extensions of this result from $L^p$ to the Hardy space ${\cal H}_1$ however impose once again a
Dini type condition on $a$.} on $L^1(\RRR^2)$ and thus the problem can not be simply
reduced to the weak-$L^1$ estimate for the Calderon-Zygmund operator $M^n$. Instead using 
the assumption that $a\in B^1_{2,1}$ we first reduce the problem to the case where in the 
multilinear expression \eqref{eq:multi} the function $a$ is replaced by its atoms
$$
M a_{k_1} M...a_{k_{n-1}} M,
$$
with  $a_k= P_k a$ and  the Littewood-Paley projection $P_k$ associated with the dyadic 
band of frequencies of size $2^k$. We then decompose 
$$
M=M_{\ge k_1} + M_{<k_1}=P_{<k_1} M + P_{\ge k_1} M
$$
and observe that $[M_{\ge k_1}, a_{k_1}]$ is a bounded operator on $L^1$. It follows that
\begin{align*}
M a_{k_1} M...a_{k_{n-1}} M &= a_{k_1} M_{\ge k_1} M...a_{k_{n-1}} M + 
[M_{\ge k_1}, a_{k_1}] M...a_{k_{n-1}} M \\ &+ M_{<k_1} a_{k_1} M...a_{k_{n-1}} M.
\end{align*}
We now proceed inductively. The first two terms can be reduced to the problem of $L^1$ estimates
for the multilinear expressions $M^2 a_{k_2}...a_{k_{n-1}} M$ and $M...a_{k_{n-1}} M$,
each containing only $(n-1)$ Calderon-Zygmund operators and $(n-2)$ atoms $a_{k_i}$.
The remaining term $M_{<k_1} a_{k_1} M...a_{k_{n-1}} M$ can be written in the form 
$$
M_{<k_1} a_{k_1} Ma_{k_2}...a_{k_{n-1}} M=
\sum_{\ell_2,...,\ell_{n-1}}M_{<k_1} a_{k_1} M_{k_1} a_{k_2} M...a_{k_{n-1}} M_{\ell_{n-1}}.
$$
The operator $M_{<k_1}$ is handled with the help of the weak-$L^1$ estimate, which
comes on one hand with a logarithmic loss but on the other hand has a certain important 
redeeming property in the choice of the constants, which in particular made dependent on 
the multi-index $\ell_1,..,\ell_n$. The remaining argument consists in
showing that the operator $M_{k_1} a_{k_2} M_{\ell_2}...a_{k_{n-1}} M_{\ell_{n-1}}$ is 
bounded on $L^1$ with the bound reflecting exponential gains in the differences of either 
of the adjacent frequencies $|\ell_m-\ell_{m-1}|$ or $|k_m-k_{m-1}|$.

The problem of $L^1$ estimates for the transport equation \eqref{eq:Tr1} with 
variable time-dependent coefficient $a(t,x)$ exemplifies  even more
the need for such multilinear 
estimates. In this case a solution  $u$ does not quite have an exponential map 
representation similar to \eqref{eq:expo}. Instead it can be written in the form
$$
u(t)=\int_0^t T\Big\{ e^{\int_s^t a(\tau) M\, d\tau}\Big\} g(s)\, ds.
$$
Here $T$ is the Quantum Field Theory (QFT)  notation for the time ordered product. Thus, we have
\begin{align}
u(t)&=\int_0^t \sum_{n=0}^\infty\frac 1{n!} 
 T\Big \{\int_s^t\int_s^t...\int_s^t a(t_1) M a(t_2)M...
a(t_n) M\, dt_1\,...\,dt_n\Big\} g(s)\, ds \nn\\ &= \int_0^t \sum_{n=0}^\infty
\int_0^t a(t_1) M dt_1 \int_0^{t_1} a(t_2) M\, dt_2....\int_0^{t_{n-1}} a(t_n) M
\int_0^{t_n} g(s)\, ds\label{eq:expo2}
\end{align}
The time ordering $T$ arranges variables $t_1,...,t_n$ in the decreasing order 
$t_1\ge t_2\ge ...\ge t_n$. Our method for deriving $L^1$ estimates for solutions 
of the transport equation \eqref{eq:Tr1} involves analyzing each of the multilinear 
expressions in the above expansion. As in the case of the time-independent coefficient
$a$ we will be able to derive an $L^1$ estimate with a logarithmic loss under the
assumption that $a$ is a $B^1_{2,1}$ valued function with an appropriate (in fact $L^1$)
time dependence.  The infinite series representation \eqref{eq:expo2} will also help us to uncover another
phenomenon. In the case when the time-dependent coefficient $a$ can be written as
a time derivative of a function $b$, i.e., $a=\pr_t b$, the $L^1$ estimate for solutions
of the transport equation \eqref{eq:Tr1} does not require Besov regularity of the 
coefficient $a$ and instead needs $L^2([0,1]; H^1)$ regularity of $a$ together with 
$L^2([0,1]; H^2)$ regularity of $b$. Our main result is the $L^1$ estimate for solutions of 
the transport equation \eqref{eq:Tr1} with the coefficient  
$a=\pr_t b +c$ with $c\in L^1([0,1];B^1_{2,1})$ and $b$ satisfying the above conditions. 

To treat this general case we consider multilinear expressions appearing in \eqref{eq:expo2}
and decompose each of the $a(t_i)$ into its Littlewood-Paley components to form a term 
$$
J_{n.\k}(t)=\int_0^t\int_0^{t_1}...\int_0^{t_n} a_{k_1}(t_1) M a_{k_2}(t_2) M...a_{k_n}(t_n) M
g(s)\, dt_1...dt_n\, ds
$$
with $\k=(k_1,...,k_n)$. For each $\k$ will be able to show the desired estimate
$$
\sup_{t\in [0,1]} \|J_{n,\k}(t)\|_{L^1(\RRR^2)} \le C N(g).
$$
The constant $C$ above  depends on the ${L^1([0,1]; H^1)}$ norms of $a_{k_i}$ and grows with $n$. 
As a consequence we face two major summation problems: first with respect to a  given multi-index 
$\k$ followed by summation in $n$. Difficulties with summation over $\k$ are connected with 
the fact that $a$ no longer has Besov regularity $B^1_{2,1}$. This lack of regularity is due to 
the term $\pr_t b$ in the decomposition of $a$. We notice however that upon substitution
into $J_n(t)$ the term $\pr_t b_{k_j}$ can be integrated by parts which results in a gain of 1/2 
derivative\footnote{The fact that the gain is only 1/2 derivative rather than the whole derivative
is due to the $L^2$ in time integrability assumption on $b$.} or, alternatively, a factor of 
$2^{-k_j/2}$. The problem however is that this gain needs to be spread across all remaining 
$(n-1)$ terms  in $J_n(t)$, which leads us to choose $k_j$ to be the highest frequency among
all $k_i$. If the highest frequency is occupied by a Besov term $c_{k_j}$, appearing the 
decomposition of $a$ we select the second highest frequency and continue the process, which 
in the end ensures summability with respect to $\k$. This analysis may potentially lead to violent
growth of the constant $C$ with respect to $n$ and extreme care is needed. We ensure that 
$C$ decays exponentially in $n$ by imposing smallness conditions on the space-time norms
of the coefficients $b$ and $c$.

We now state our result precisely.
Consider the transport equation
$$
\pr_t u- a(t,x) M u = g(t,x),\qquad u(0,x)=0.
$$
We assume that for the coefficient $a$ 
\bea
\|a\|_1:=\|a\|_{L_t^2H^1}=(\int_0^1\|a(t)\|_{H^1(\RRR^2)}^2)^{1/2}
\le \De_0.\label{eq:Tr2}
\eea
In  addition $a$ can be decomposed as follows,
\bea
a=\pr_t b+ c\label{eq:Tr3}
\eea
where,
\bea
\|b\|_2&:=& (\int_0^1\| b(t)\|_{H^2(\RRR^2)}^2+\int_0^1\|
\pr_t b(t)\|_{H^1(\RRR^d)}^2 )^{1/2}\le \De_0\label{eq1-norm}\\
\|c\|_3&:=&\int_0^1\|c(t)\|_{B^1_{2,1}(\RRR^2)}dt\le\label{eq2-norm}
\De_0
\eea
with $B^{1}_{2,1}(\RRR^2)$  the classical inhomogeneous Besov space 
defined by the  norm,
\beaa
\|v\|_{B^{1}_{2,1}(\RRR^2)}=\|P_{\le 0} v\|_{L^2}+
\sum_{k\in\ZZZ_+}2^k\|P_k v\|_{L^2(\RRR^2)}.
\eeaa
The operator $M$ is the classical translation invariant Calderon-Zygmund operator
on $\RRR^2$, given by the symbol $m(\xi)$ verifying 
\be{eq:reg}
|\pr^\a m(\xi)|\le c (1+|\xi)^{-|\alpha|},\qquad \forall \xi\in \RRR^2. 
\end{equation}
We prove the following theorem,
\begin{theorem}[Main Theorem] Under the above assumptions ,
if $\De_0$ is sufficiently small, we have the estimate,
\bea
\sup_{t\in[0,1]}\|u(t)\|_{L^1(\RRR^2)}\les  C N(g)\label{eq:Thm1}
\eea
where,
\begin{equation}\label{eq:AA}
N(g)=\|g\|_{L^1([0, 1]\times \RRR^2)}\,\,\log^{+} \big{\{}
 \|<x>^3 g\|_{L^\infty([0,1]\times\RRR^2)}\big{\}} + 1.
\end{equation}
\label{thm:main}
\end{theorem}
\begin{remark}
For a function $g$ of compact support the expression $N(g)$ can be 
controlled as follows
\begin{equation}\label{eq:Ng}
N(g)\les \|g\|_{L^1([0, 1]\times \RRR^2)}\,\,\log^{+}
\|g\|_{L^\infty([0,1]\times\RRR^2)}+1
\end{equation}
\end{remark}
\begin{remark}
Condition \eqref{eq:reg} implies that the symbol of the operator $M$ is smooth 
at the origin, which in principle eliminates a large class of Calderon-Zygmund operators
from our consideration. We argue however that this condition is not particularly restrictive
and can be replaced with assumptions of additional spatial decay on the coefficients
$a(t,x)$. Moreover, in our application (see the paragraph below)  we consider the corresponding transport equation 
on a compact manifold (2-sphere) instead of $\RRR^2$, where the existence of a spectral 
gap ensures that condition \eqref{eq:reg} holds. In that context a prototype for 
$M$ is the operator $(-\Delta)^{-1} \nab^2$. Moreover, in that case $N(g)$ can be 
replaced by the $L\log L$ type expression \eqref{eq:Ng}. 
\end{remark}
The above theorem is a  vastly simplified model case for the type of
result we need in \cite{KR6} to prove  a conditional 
regularity result for the Einstein vacuum equations.
The  main assumption in \cite{KR6}, concerning 
the pointwise  boundedness of  the  deformation  tensor
of the unit, future, normal vectorfield  to a  space-like foliation,
allows us  to   bound  the flux of the space -time  curvature through
the boundary
$\NN^{-}(p)$ of the causal past of any point $p$ of the space-time
under consideration.  In \cite{KR1}--\cite{KR4}, see also \cite{Q}, 
we were able to show  that the boundedness of the flux of curvature
through  $\NN^{-}(p)$ suffices to control the radius of injectivity
of $\NN^{-}(p)$. This result, together with the  construction
 of a  first order parametrix in 
\cite{KR5},  is  used  in \cite{KR6} to derive 
pointwise bounds for the curvature tensor of the 
corresponding spacetime. To control the  main error term
generated by the parametrix one needs however to bound 
the $L^1$ norm of the  first two tangential  derivatives of 
$\trch$ along $\NN^{-}(p)$,  with $\trch $ the
trace of the null second fundamental 
form of $\NN^{-}(p)$.
One can show that the second tangential derivatives
of $\trch$ verifies a transport equation along the null geodesic
generators of $\NN^{-}(p)$ which can be modeled, very  roughly,
 by   \eqref{eq:Tr1}, with $g$ a term  whose
$L^1$ norm along $\NN^{-}(p)$ is bounded by the flux of curvature .
In fact a more realistic model would be  to consider a transport,
similar to
\eqref{eq:Tr1},  along the null geodesics of a 
past  null cone $\NN^{-}(p)$ in Minkowski
 space $\RRR^{3+1}$ with $t$  denoting  the value
of the standard afine parameter along null geodesics
and $x=(x^1, x^2)$ denoting the standard sperical coordinates
on the 2-spheres  $S_t$, corresponding to constant value
of $t$ along $\NN^{-}(p)$. Thus the singular 
integral operator $M$  would  act on $S_t$.

Finally we believe that our result, or rather
our proof of the result,  can be applied to
other situations where one  needs to  make
$L^1$ or $L^\infty$ estimates for singular
 transport equations,  where  a simple logarithmic 
loss is unavoidable.

\section{Preliminary results} 
 We recall briefly  the classical    Littlewood-Paley   decomposition
of functions  defined on ${\Bbb R}^d$,  \beaa
f=f_0+ \sum_{k\in
\Bbb Z_+} f_k
\eeaa 
 with frequency localized    components $f_k$, i.e. 
$\widehat{f_k}(\xi)=0$  for all values of $\xi$ outside the annulus
 $2^{k-1}\le |\xi|\le 2^{k+1}$ and a function $f_0$ with frequency localized in the ball $|\xi|\le 1$. 
 Such a decomposition can be  easily  achieved 
by choosing a  test function $\chi=\chi(|\xi|)$ in  Fourier 
 space,  supported in $\f12\le |\xi|\le 2$, and such that, for all $\xi\neq 0$, 
$\sum_{k\in \Bbb Z}\chi(2^{-k}\xi)=1$. Then for $k>0$ set $\widehat{f_k}(\xi)=
\chi(2^k\xi) \hat{f}(\xi)$ or, in physical space,
$$P_k f= f_k=p_k*f$$
where $p_k(x)= 2^{nk}p(2^{k} x)$ and $p(x)$ the 
inverse Fourier transform of $\chi$, while 
$$
\hat f_0 (\xi) = \left (1-\sum_{k\in \ZZZ_+} \chi(2^{-k}\xi)\right ) \hat f(\xi)
$$
and $f_0=P_0 f$.
 The operators $P_k$ 
are called cut-off operators or, somewhat improperly, Littlewood-Paley projections.

Let $M$ be a Calderon-Zygmund operator  with multiplier $m$, i.e.,
\bea
\widehat{ Mf}(\xi)=m(\xi) \hat{f}(\xi)\label{eq:CZ1}
\eea
Here $m$ is a smooth function satisfying
\bea
|\pr_\xi^{\a} m(\xi)|\le c (1+|\xi|)^{-|\a|},\qquad
\forall \xi\in \RRR^d\label{eq:const-c}
\eea
for all multiindices $\a$ with $|\a|\le d+6$ and a fixed constant
$c>0$. According to Michlin-H\"ormander theorem we have,
\bea
| m(x)|\le c |x|^{-d},\qquad | \pr_x m(x)|\le c |x|^{-d-1}
\eea
Due to the smoothness of the symbol of $M$ at the origin we can also 
add the estimate
\be{eq:smooth}
|m(x)|\le c (1+|x|)^{-d-6}
\end{equation}
 We shall make use of
the standard Calderon-Zygmund estimates in $L^p$, $1<p<\infty$,
\beaa
\|Mf\|_{L^p}\le C_p \|f\|_{L^p}
\eeaa
as well as the weak-$L^1$ estimate
$$
|\{x:\,|Mf(x)|>\la\}\le      C\la^{-1} \|f\|_{L^1}
$$ 
Our first result is a global version of the standard local $L^1$ estimate for 
a multiplier $M$.  The local estimate in a ball 
$B_R$  does not require the condition
\eqref{eq:smooth} and takes  the form
$$
\|Mf\|_{L^1(B_R)}\le C_R (\|f\|_{L^1}\log^+\|f\|_{L^\infty}+1).
$$
We have the following
\begin{lemma}
Let $M$ be a multiplier satisfying \eqref{eq:smooth}. Fix an $L^1(\RRR^d)$ positive function $\b$
and a constant $\mu>0$. 
Then for any smooth function $f$ of compact support
$$
\|Mf\|_{L^1}\le C N_{\mu,\b}(f),
$$
where 
$$
N_{\mu,\b}(f)= \mu \|\b\|_{L^1} + \|f\|_{L^1} \log^+\{\sup_{{\bf a}\in\ZZZ^d} \frac{\sum_{|{\bf b}-{\bf a}|\le 3    }\|\chi_{\bf b} f\|_{L^\infty}}{\mu
\|\chi_{\bf a} \b\|_{L^1}}\},
$$
$\chi_{\bf a}$ is a partition of unity adapted to the balls of radius one with 
centers at integer lattice points ${\bf a}$ and 
$\log^+x=\log(2+|x|) $.
\end{lemma}
\begin{proof}
We first note that the problem can be reduced to the case when the kernel of $M$,
given by the function $m(x)$, has compact support. This follows since
$$
M f(x) = M_0 f (x) + M_1 f(x),\qquad M_1 f(x) =\int \chi(x-y) m(x-y) f(y) \, dy,
$$
where $\chi$ is a smooth cut-off function vanishing on the ball of radius one. 
Assumption \eqref{eq:smooth} guarantees that $\chi(x) m(x)$ is integrable.
 As a consequence, 
$$
 \|M_1 f\|_{L^1}\le C\|f\|_{L^1}. 
$$
To deal with $M_0$ we proceed in the usual fashion by writing 
\begin{align*}
\|M_0 f\|_{L^1} &= \int_0^\infty |\{x: |M_0 f(x)|>\lambda\}|\, d\lambda\le  
 \int_0^\infty |\{x: |M_0 f_{<\la} (x)|>\lambda\}|\, d\lambda\\& +  
 \int_0^\infty |\{x: |M_0 f_{\ge \lambda}(x)|>\lambda\}|\, d\lambda,
\end{align*}
where $f_{<\la}(x)$ is the function coinciding with $f(x)$ on the set where 
$|f(x)|<\la$ and vanishing on its complement, and $f_{\ge \la}=f(x)-f_{<\la}$. 
To estimate the term with $f_{<\la}$ we use the weak-$L^2$ estimate
\begin{align*}
 \int_0^\infty |\{x: |M_0 f_{<\la} (x)|>\lambda\}|\, d\lambda\le 
 C \int_0^\infty \frac {\|f_{<\la}\|_{L^2}^2}{\la^2} &= C \int\int_{|f(x)|}^\infty \la^{-2} |f(x)|^2\, d\la\, dx\\ &=
  C \int |f(x)|\,  dx
\end{align*}
\def\aa{{\bf a}}
To estimate the term with $f_{\ge \la}$ we decompose $f_{\ge \la}$ into 
the sum of functions $f_{\ge \la}^\aa=\chi_\aa f_{\ge \la}$
$$
f_{\ge \la}=\sum_{\aa\in{\Bbb Z}^d} \chi_\aa f_{\ge\la},
$$
where $\chi_{\aa}$ is a partition of unity, parametrized by integer lattice points in $\RRR^d$ 
with the property that  the support of $\chi_\aa$ is contained in the ball of radius two 
around the point $\aa\in \RRR^d$. Since the kernel of $M_0$ is supported in a ball of radius
one,  the support of $M_0 f^\aa_{\ge \la}$ is contained in the ball of radius three around $k$. 
As a consequence, there are at most $3^d C$ functions  $M_0 f^\aa_{\ge \la}$ containing
any given point $x$ in their support. Therefore,
$$
 |\{x: |M_0 f_{\ge \lambda}(x)|>\lambda\}|\le \sum_{\aa\in{\Bbb Z}^d}
  |\{x: |M_0 f^\aa_{\ge \lambda}(x)|> \lambda (3^dC)^{-1}\}|.
$$
We also  have the  trivial estimate, with another constant still denoted  $C$, 
$$
  |\{x: |M_0 f^\aa_{\ge \lambda}(x)|> \lambda (3^dC)^{-1}\}|\le 3^d C.
$$
Thus, using a weak-$L^1$ estimate we obtain
\beaa
J_{\bf a}:&=& \int_0^\infty |\{x: |M_0 f^\aa_{\ge \lambda}(x)|>\lambda (3^dC)^{-1}\}|\, d\lambda\\
&\le &
 \int_0^{\la_0} 3^d C +   3^d C \int \int_{\la_0}^\infty \la^{-1} \|\chi_\a f_{\ge\la}\|_{L^1}\, d\la\\& 
 \le &3^d C \la_0 + 3^d C \int_{\la_0}^\infty \int_{|f(x)|\ge \la } \la^{-1} |\chi_\aa f(x)|\,dx\, d\la\\ &\le &
 3^d C\la_0 +  3^d C \int  \chi_\aa(x) |f(x)|  \,\big| \log \frac {|f(x)|}{\la_0} \big|\, dx\\
 &\les & 3^d C\la_0 +  3^d C \int_{|f(x)|\ge \la_0}  \chi_\aa(x) |f(x)|  \, \log \frac {|f(x)|}{\la_0} \, dx\\
 &\les & 3^d C\la_0 +  3^d C \int  \chi_\aa(x) |f(x)|  \, \log^+ \frac {|f(x)|}{\la_0} \, dx
\eeaa
for some $\la_0>0$. We now choose $\la_0=\mu \int \chi_\aa(x) \b(x)\, dx$. The above estimate
then becomes
\beaa
J_{\bf a}&\le &
 3^d C\left (\mu \|\chi_{\bf a}\b\|_{L^1} + \int \chi_\aa(x) |f(x)| \,
 \log^+
 \frac{|f(x)|}{\mu \|\chi_{\bf a} \b\|_{L^1}}   \right).\\
 &\les&3^d C\left (\mu \|\chi_{\bf a}\b\|_{L^1} + \int |f(x)|
  \chi_\aa(x) \,  |\log^+ \sum_{\bf b}\frac{\chi_{\bf b}(x)|f(x)|}{\mu\|\chi_a\b\|_{L^1}}  \right)\\
   &\les&3^d C\left (\mu \|\chi_{\bf a}\b\|_{L^1} + \int |f(x)|\chi_{\bf a}(x)
 \,  \log^+ \sum_{|{\bf b}-{\bf a}|\le 3}\frac{\chi_{\bf b}(x)|f(x)|}{\mu\|\chi_a\b\|_{L^1}}  \right)\\
   &\les&3^d C\left (\mu \|\chi_{\bf a}\b\|_{L^1} + \int |f(x)|\chi_{\bf a}(x)
 \,  \log^+ \sum_{|{\bf b}-{\bf a}|\le 3}\frac{\|\chi_{\bf b}(x)|f(x)|   }{\mu\|\chi_a\b\|_{L^1}}  \right)\\
  &\les&3^d C\left (\mu \|\chi_{\bf a}\b\|_{L^1} +\|f\chi_{\bf a}\|_{L^1}
 \,  \log^+\sup_{{\bf a}\in\ZZZ^d} \sum_{|{\bf b}-{\bf a}|\le 3}\frac{\|\chi_{\bf b}f\|_{L^\infty}   }{\mu\|\chi_a\b\|_{L^1}}  \right)\
\eeaa
Now,
\beaa
\|M_0 f\|_{L^1}&\les &
 \int_0^\infty |\{x: |M_0 f_{<\la} (x)|>\lambda\}|\, d\lambda +  
 \int_0^\infty |\{x: |M_0 f_{\ge \lambda}(x)|>\lambda\}|\, d\lambda\\
 &\les& C\|f\|_{L^1}+\sum_{\bf a\in\ZZZ^d}J_{\bf a}\\
 &\les& C\|f\|_{L^1}+3^dC\left(\mu \|\b\|_{L^1}+\|f\|_{L^1}\log^+\sup_{{\bf a}\in\ZZZ^d} \sum_{|{\bf b}-{\bf a}|\le 3}\frac{\|\chi_{\bf b}f\|_{L^\infty}   }{\mu\|\chi_a\b\|_{L^1}}  \right)
\eeaa
as desired.

\end{proof}
We also need to consider powers of $M^n$ of $M$
with multipliers $m^{(n)}(\xi)=m(\xi)^n$. Clearly, 
there exists a constant $C>0$ depending only on
$c$ and $d$ such that,
\bea
| m^{(n)}(x)|\le C^n |x|^{-d},\quad
 | \pr_x m^{(n)}(x)|\le C^n |x|^{-d-1},\quad |m^{(n)}(x)|\le C^n (1+|x|)^{-d-6}
\eea
Thus, for a similar  $C>0$,
\bea
\|M^nf\|_{L^1}\le C^n N_{\mu,\b}(f)
\label{eq:CZ-log-n}
\eea

 Let 
$m_k(\xi)=\chi(2^k\xi)m(\xi) $ and  denote by $M_k$
the operator defined by the multiplier $m_k$. Clearly
$M_k f=P_k( Mf)$. We shall also denote by $M_J$  the operator
$P_J M$ with multiplier $m_J=\sum_{k\in J} m_k$ for
any interval $J\subset\ZZZ$. In physical space,
\beaa
M_kf(x)=\int_{\RRR^d}m_k(x-y)f(y) dy,\quad M_{\ge k} f=
\int_{\RRR^d}m_{\ge k}(x-y)f(y) dy
\eeaa
 We have  the following,
\begin{lemma}\label{le:L1} Let $k\in \ZZZ_+\cup \{0\}$ and 
assume that $a_k$ is a function whose frequency 
is  supported in the band $2^{k-1}\le |\xi|\le 2^{k+1}$, or in the case $k=0$ in the ball $|\xi|\le 1$.
Then, there exists a constant $C>0$ such that for all $n\in\NNN$,
\beaa
\|[(M^n)_{\ge k}\,,\, a_k] f\|_{L^1}\le C^n
\|a_k\|_{L^\infty}\|f\|_{L^1}
\eeaa
\end{lemma} 
\begin{proof}:\quad We have,
\beaa
 C(a_k) f:&=&(M^n)_{\ge k}(a_k f)(x)-a_k(x) (M^n)_{\ge k}f(x\\
&=&
\int m^{(n)}\,_{\ge k}(x-y)\big(a_k(y)-a_k(x)\big) f(y)dy
\eeaa
To show that the the integral operator $C(a_k)$ maps
$L^1$ into $L^1$ it suffices to show that,
\beaa
I&=&\sup_{y} I(y)\\
 I(y)&=&  \int |m^{(n)}\,_{\ge k}(x-y)||a_k(y)-a_k(x)|dx\le C^n
\|\a_k\|_{L^\infty}
\eeaa
We write,
\beaa
I(y)&\le &I_1(y)+I_2(y)\\
I_1(y)&=&
\int_{|x-y|\ge 2^{-k}} |m^{(n)}\,_{\ge k}(x-y)||a_k(y)-a_k(x)|dx\\
I_2(y)&=&\int_{|x-y|\le 2^{-k}}| m^{(n)}\,_{\ge
k}(x-y)||a_k(y)-a_k(x)|dx
\eeaa
We have,
\beaa
|a_k(y)-a_k(x)|\le |x-y|\sup_{z\in[x,y]} |\pr a_k(z)|\les 2^k|x-y|\,
\|a_k\|_{L^\infty}
\eeaa
We also have, 
\beaa
|m^{(n)}\,_{\ge k}(x)|\le C^n  |x|^{-d}
\eeaa
Thus,
\beaa
I_2(y)&\le & C^n\|a_k\|_{L^\infty}
 \int_{|x-y|\le 2^{-k}}|x-y|^{-d}2^k|x-y|dx\les C^n
\|a_k\|_{L^\infty}
\eeaa
Also, since,
\beaa
|m^{(n)}\,_{\ge k}(x)|\le C^n 2^{-k} |x|^{-d-1}
\eeaa
\beaa
I_1(y)&\le& C^n\|a_k\|_{L^\infty}
\int_{|x-y|\ge 2^{-k}}  2^{-k}|x-y|^{-d-1}dx\les C^n
\|a_k\|_{L^\infty}
\eeaa
as desired.
\end{proof}
We shall now prove the following,
\begin{proposition}Let $M$ be a Calderon-Zygmund operator  on $\RRR^2$ with the symbol
satisfying \eqref{eq:const-c} and
$a=a(x)$  a smooth function verifying the bound,
\bea
\|a\|_{B^1_{2,1}(\RRR^2)}\le  A
\eea
Then, for every  positive integer $n$ we have,
\bea
\|(a M)^n f\|_{L^1}\le C^n A^n N(f)
\eea
with $N(f)$ defined by \eqref{eq:Thm1}.
\end{proposition}
\begin{remark}\quad Observe that the proposition remains 
valid if we replace $(aM)^n$ by $a_{(1)}M_{(1)} a_{(2)} M_{(2)}\ldots a_{(n)} M_{(n)}$
with 
\beaa
\|a_{(i)}\|_{B^1_{2,1}(\RRR^d)}\le  A,\qquad i=1,\ldots n
\eeaa
and $M_1,M_2,\ldots M_n$ translation invariant Calderon-Zygmund operators with
symbols which  are uniformly bounded by the same constant $c$,
 see \eqref{eq:const-c}.
\end{remark}
 The proof follows immediately from the following lemma.
\begin{lemma}\label{le:main} Let $(k_1,...,k_n)$ be an n-tuple of non-negative 
integers and assume that  the functions  $a_{k_i}$ with  $0\le i\le
n$ have frequencies supported in the dyadic shells  $[2^{k_{i-1}}, 2^{k_{i+1}}]$, or
in the case $k_i=0$ in the ball $|\xi|\le 1$.
Then for some positive constant $B$,
\bea
\|Ma_{k_1}M\ldots a_{k_n}M f \|_{L^1}&\les&  B^n 
A_{k_1\ldots k_n}N(f)\label{eq:1.2.4}
\eea
where
\bea
A_{k_1\ldots k_n}= \|a_{k_1}\|_{H^1}\cdots
\|a_{k_n}\|_{H^1} \label{eq:1.2.5}
\eea
\end{lemma}
\begin{proof}:\quad 
We prove   by induction on $n$  the following stronger
version of estimate \eqref{eq:1.2.4},
\bea
\|M^{l} a_{k_1}M\ldots a_{k_n}M f \|_{L^1}\les  B_1^{n+l} B_2^n 
A_{k_1\ldots k_n}
N(f)\label{eq:1.2.4-again}
\eea
with appropriately chosen constants constants $B_1, B_2$.
Assume that the estimate has been proved for $(n-1)$
 and any $l\in\NNN$.
Splitting $\bar M:=M ^{l}=\bar M_{<k_1}+\bar M_{\ge k_1}$  we
need to prove, 
\bea
\|\bar M_{\ge k_1}(a_{k_1}Ma_{k_2}\ldots a_{k_n}M )f\|_{L^1}\les
 B_1^{n+l} B_2^n  A_{k_1\ldots
k_n}N(f)\label{eq:1.2.7}\\
\|\bar M_{< k_1}(a_{k_1}Ma_{k_2}\ldots a_{k_n}M) f\|_{L^1}\les
 B_1^{n+l} B_2^n  A_{k_1\ldots
k_n}N(f)\label{eq:1.2.8}
\eea
To deal with the first inequality we write,
\beaa
\bar M_{\ge k_1}a_{k_1}Ma_{k_2}\ldots a_{k_n}M&=&a_{k_1}\bar M_{\ge
k_1}M a_{k_2}\ldots a_{k_n}M \\&+&[\bar M_{\ge
k_1},a_{k_1}]Ma_{k_2}\ldots a_{k_n}M
\eeaa
According to Lemma \ref{le:L1} and the Bernstein inequality $\|a_k\|_{L^\infty}\les \|a_k\|_{H^1}$, 
we have,
\beaa
\|[\bar M_{\ge k_1},a_{k_1}]Ma_{k_2}\ldots a_{k_n}Mf\|_{L^1}\les
 C^{l} \|a_{k_1}\|_{H^1}
\|Ma_{k_2}\ldots a_{k_n}Mf\|_{L^1}
\eeaa
Also,
\bea
\|a_{k_1}\bar M_{\ge k_1} M a_{k_2}\ldots
a_{k_n}Mf \|_{L^1}&\les& \|a_{k_1}\|_{L^\infty}
\|M^{l+1}a_{k_2}\ldots a_{k_n}Mf\|_{L^1}
\eea
Thus, taking into account our induction
hypothesis,
\beaa
\|M_{\ge k_1}(a_{k_1}Ma_{k_2}\ldots a_{k_n}M )f\|_{L^1}&\les &
C^{l}\|a_{k_1}\|_{H^1}\cdot\|Ma_{k_2}M\ldots a_{k_n}M f\|_{L^1}\\
&+& \|a_{k_1}\|_{H^1}
\|M^{l+1}a_{k_2}\ldots a_{k_n}Mf\|_{L^1}\\
&\les& (C^{l} B_1^{n} B_2^{n-1}+B_1^{n+l} B_2^{n-1})A_{k_1\ldots k_n}N(f)\\
&\les& B_1^{n+l} B_2^n A_{k_1\ldots k_n}N(f)
\eeaa
as desired, provided that the constants $B_1, B_2$ are sufficiently large, in fact  we need $ B_1\ge C$ and $B_2\ge 1$.

We now consider the more difficult  term
\beaa
\bar M_{< k_1}(a_{k_1}Ma_{k_2}\ldots a_{k_n}M) f=
\bar M_{< k_1}\big(a_{k_1}M(g)\big)=\bar M_{<
k_1}\big(a_{k_1}M_{k_1}(g)\big)
\eeaa
with $g=(a_{k_2}Ma_{k_3}\ldots a_{k_n}M) f$. 
Note that if $k_1=0$ the operator $\bar M_{<k_1}$ is a multiplier with a smooth symbol
of compact support. As a consequence it is bounded on $L^1$ and,
with $a_0=a_{k_1}$,
\begin{align*}
\|\bar M_{< 0}(a_{0}Ma_{k_2}\ldots a_{k_n}M) f\|_{L^1}&\le C^l
\|a_{k_1}\|_{H^1} \|Ma_{k_2}\ldots a_{k_n}M) f\|_{L^1}\\ &\les
C^l B_1^{n} B_2^{n-1} A_{k_1...k_n} N(f).
\end{align*}
Therefore to prove
\eqref{eq:1.2.8} we need to consider the case $k_1>0$ and
 estimate,
 \beaa
\|\bar M_{< k_1}(a_{k_1}Ma_{k_2}\ldots a_{k_n}M f) \|_{L^1}
\eeaa
We further decompose as follows,
\bea
\bar M_{< k_1}(a_{k_1}Ma_{k_2}\ldots a_{k_n}M f)&=&
\sum_{[l]_n}\bar M_{< k_1}M_{[k]_n, [l]_n}(f)
\\
    M_{[k]_n, [l]_n}(f)&=&
 a_{k_1}M_{l_1}a_{k_2}\ldots M_{l_{n-1}}
a_{k_n}M_{l_n} f\nn
\eea
with $[l]_n$ denoting an  arbitrary integer n-tuple $(l_1,...,l_n)\in ({\Bbb Z}_+\cup \{0\})^n$ and $[k]_n=(k_1,\ldots ,k_n)$.  Whenever there is no possibility of confusion we shall drop the index $n$ and write simply  simply write $[k], [l]$.
By the  triangle inequality
\beaa
\|\bar M_{< k_1}(a_{k_1}Ma_{k_2}\ldots a_{k_n}M f) \|_{L^1}\le 
\sum_{[l]_n}
\|\bar M_{< k_1}M_{[k]_n, [l]_n}(f) \|_{L^1}
\eeaa
We note that in the expression $\bar M_{< k_1} a_{k_1}M_{l_1} (a_{k_2}\ldots a_{k_n}M_{l_n} f)$
the frequency $l_1$ is forced to be of the order of $k_1$. This allows us to insert a factor of 
$2^{-|k_1-l_1|}$ in the above expression.
Using \eqref{eq:CZ-log-n} we then derive,
\bea
\|\bar M_{< k_1}M_{[k], [l]}(f)\|_{L^1}\les 
2^{-|k_1-l_1|}
B_1^{l} B_2 N_{\mu(\,[l]\,),\b}\big(\,M_{[k],[l]}( f)\, \big)\label{eq:onemore}
\eea
Here, the notation $\mu(\,[l]\,)$ indicates that the  scalar  $\mu$ will be chosen dependent  on the multi-index $[l]=[l]_n$. Recall that\footnote{For simplicity
of notation we drop  the summation $\sum_{|{\bf b}-{\bf a}|\le 3 }$  which will 
only adds a finite number of terms of the same type. },
$$
N_{\mu,\b}(g)=\mu \|\b\|_{L^1} + \|g\|_{L^1}
\log^+\{\sup_{{\bf a}\in\ZZZ^d} \frac{\|\chi_{\bf a} g\|_{L^\infty}}{\mu
\|\chi_{\bf a} \b\|_{L^1}}\}
$$
We now make the following choice for the scalar $\mu$. The choice will be justified in the lemmas below.
\beaa
\mu([{ l}])&=&A_{k_1\ldots k_n} 2^{-\alpha([{ l}]_n)},\\
\alpha([{ l}]&=&\frac 12 \sum_{m=2}^{n} \min \Big (|l_{m}-l_{m-1}|, |\l_m-k_m|\Big)
\eeaa 
 We also choose  the function 
 \beaa
 \b=(1+|x|)^{-3}.
 \eeaa
 Observe that  the following holds true,
\bea
\bigg( \frac{<\bf b>}{< \bf a>}\bigg)^{-3} \|\chi_{\bf b} \b\|_{L^1}\le
\|\chi_{\bf a} \b\|_{L^1}\le\bigg ( \frac{<\bf b>}{< \bf a>}\bigg)^{3}
 \|\chi_{\bf b} \b\|_{L^1}
 \eea
We will need to make use of the following,
\begin{lemma} The following estimates hold true for the expression,
$$M_{[k],[l]}(f)=a_{k_1} M_{l_1} a_{k_2}....a_{k_n} M_{l_n} f,$$
\bea
\|M_{[k],[l]}(f)\|_{L^1}& \les &  C^{n}  2^{-2\alpha([l]_n)}
A_{k_1..k_n}\|f\|_{L^1} \label{eq:1.2.14}\\
\|\chi_{\bf a}M_{[k],[l]}(f)\|_{L^\infty} &\les& C^n A_{k_1..k_n}\sum_{{\bf b}\in \ZZZ^2}< |{\bf b}-{\bf a}|>^{-3} 
\|\chi_{\bf b} f\|_{L^\infty} \label{eq:1.2.15} 
\eea
\label{le:main}
\end{lemma}
We postpone the proof of the  lemma to the end of this section.

Now, using \eqref{eq:onemore}
\beaa
&&\|\bar M_{< k_1}(a_{k_1}Ma_{k_2}\ldots a_{k_n}M f) \|_{L^1}\le 
\sum_{[l]_n}
\|\bar M_{< k_1}M_{[k], [l]}(f) \|_{L^1}\\
&&\les\sum_{[l]} 2^{-|k_1-l_1|} \left (\mu(\, [l])\, \|\b\|_{L^1} +\|M_{[k],[l]} (f)\|_{L^1}
\log^+\{\sup_{{\bf a}\in\ZZZ^d} \frac{\|\chi_{\bf a}M_{[k],[l]}( f)\|_{L^\infty}}{\mu(\,[l]\,)
\|\chi_{\bf a} \b\|_{L^1} } \}\right)
\eeaa
Given our choice of $\mu(\,[l]\,)$ we have,
\beaa
\sum_{\,[l]\,} 2^{-|k_1-l_1|} \mu(\,[l]\,)& =&A_{k_1...k_n} 
\sum_{[\,l\,]} 2^{-|k_1-l_1|} 2^{-\alpha(\,[l])\,} \\ &=& 
A_{k_1...k_n} 
\sum_{[l]} \left (2^{-|k_1-l_1|}\c 2^{-\frac 12\min \left (|l_{2}-l_{1}|, |l_2-k_2|\right)}\c\ldots\c
2^{-\frac 12\min \left (|l_{n}-l_{n-1}|, |l_n-k_n|\right)}\right)\\
&\les& A_{k_1...k_n} 
\eeaa
Thus, in order to end he proof of \eqref{eq:1.2.8}
it suffices to  
show that  
 \bea
\sum_{[l]} 2^{-|k_1-l_1|}  \|M_{[k],[l]} (f)\|_{L^1}
\log^+\{\sup_{{\bf a}\in\ZZZ^d} \frac{\|\chi_{\bf a}M_{[k],[l]}( f)\|_{L^\infty}}{\mu(\,[l]\,)
\|\chi_{\bf a} \b\|_{L^1} } \}\les C^n A_{k_1\ldots k_n} N(f)
 \eea
 Using \eqref{eq:1.2.14} and \eqref{eq:1.2.15} 
 and recalling the definition  of $\mu[l]$, $\b(x)$, we obtain
\beaa
&&\sum_{[l]} 2^{-|k_1-l_1|}  \|M_{[k],[l]} (f)\|_{L^1}
\log^+\{\sup_{{\bf a}\in\ZZZ^d} \frac{\|\chi_{\bf a}M_{[k],[l]}( f)\|_{L^\infty}}{\mu(\,[l]\,)
\|\chi_{\bf a} \b\|_{L^1} } \}\\
&&\les
C^{n} A_{k_1...k_n} \sum_{[l]} 2^{-|k_1-l_1|} 2^{-2\alpha([l])}
\|f\|_{L^1} \,\log^+\Big{\{}C^{n-1}  \sup_{{\bf a}\in\ZZZ^d}  \sum_{{\bf b}\ne{\bf a}}
 <|{\bf b}-{\bf a}|>^{-3} \frac{2^{\alpha([l])} \|\chi_{\bf b} f\|_{L^\infty}}{
\|\chi_{\bf a} \b\|_{L^1}}\Big{\}}\\
&&\les 
C^{2n} A_{k_1...k_n}  \sum_{[l]} 2^{-|k_1-l_1|} 2^{-\alpha([l])}
\|f\|_{L^1} \,\log^+\Big{\{} \sup_{{\bf a}\in\ZZZ^d}  \frac{\|\chi_{\bf a} f\|_{L^\infty}}{
\|\chi_{\bf a} \b\|_{L^1}}\Big{\}}\\
&&\les C^{2n} A_{k_1...k_n} 
\|f\|_{L^1} \,\log^+\Big{\{} \sup_{{\bf a}\in\ZZZ^d}  <|{\bf a}|>^3 {\|\chi_{\bf a} f\|_{L^\infty}}\Big{\}}\\
&&\les C^{2n} A_{k_1...k_n} N(f),
\eeaa
as desired.
 Here we  have used,
  $$ (1+|{\bf a}|)^3 \les (1+|{\bf b}-{\bf a}|)^{3}(1+|{\bf b}|)^{3} $$
  and the  finiteness of the sum
$$
\sum_{[l]} 2^{-|k_1-l_1|} 2^{-\alpha(\,[l]\,)}=
\sum_{[l]} \left (2^{-|k_1-l_1|} 2^{-\frac 12\min \left (|l_{2}-l_{1}|, |l_2-k_2|\right)}\c\ldots\c
2^{-\frac 12\min \left (|l_{n}-l_{n-1}|, |l_n-k_n|\right)}\right)
$$ 
It remains to prove Lemma \ref{le:main}.
Estimate  \eqref{eq:1.2.14} follows recursively provided that we can 
establish the following
\bea
\|M_{l_{m-1}} a_{k_m} P_{l_m} h\|_{L^1}\les \|a_{k_m}\|_{H^1} 2^{-\min(|l_m-l_{m-1}|,
|l_m-k_m|)} \|h\|_{L^1} 
\label{eq:rec-main}
\eea
In fact, since $M_{l_{m-1}}$ is bounded in $L^1$,
it suffices to prove,
\bea
\|P_{l_{m-1}} a_{k_m} P_{l_m} h\|_{L^1}\les \|a_{k_m}\|_{H^1} 2^{-\min(|l_m-l_{m-1}|,
|l_m-k_m|)} \|h\|_{L^1} 
\label{eq:rec-main'}
\eea
On the other hand, estimate  \eqref{eq:1.2.15} is a localized version of the trivial estimate
$$
\|a_{k_1} M_{l_1} a_{k_2}....a_{k_n} M_{l_n} f\|_{L^\infty}\les   C^{n} A_{k_1..k_n}
\|f\|_{L^\infty},
$$
which holds since each of the frequency localized Calderon-Zygmund operators 
$M_{l}$ are bounded on $L^p$ including $p=1,\infty$.
 Its 
localized version
follows inductively from the estimate,
\bea
\|\chi_{\bf a}M_{l}\chi_{\bf b}g\|_{L^\infty}\le C(1+|{\bf b}-{\bf a}|)^{-3}\|g\|_{L^\infty},\qquad l\ge 0
\eea
which holds true  on account of the sharp localization of the kernel of $M_l$,  
in physical space,  due to the smoothness of the symbol of $M$ at zero. 
 Indeed the kernel of $m(x-y)$
   of  the
 operator $\chi_{\bf a}M_{l}\chi_{\bf b}$ verifies,
 \beaa
 |m(x-y)|\le C\chi_{\bf a}(x)(1+|x-y|)^{-6}\chi_{\bf b}(y)\le C
 (1+ |{\bf b}-{\bf a}|)^{-3} m_1(x-y)
 \eeaa
 with $m_1(x-y)=(1+|x-y|)^{-3}$ in $L^1$.
 
To prove \eqref{eq:rec-main'} we distinguish the following cases.
\begin{enumerate}
\item Assume  $ l_{m-1}<k_m$.
 Observe that $P_{l_{m-1}}(a_{k_m}P_{l_m} h)=0$ unless
$|l_m-k_m|\le 2$.
 Therefore,   since 
 \beaa
 \min\big(|l_m-l_{m-1}|,
|l_m-k_m|)\big)\approx 1
\eeaa
  we have , 
\beaa
\|P_{l_{m-1}}(a_{k_m}P_{k_m} h)\|_{L^1}&\les &
 \|a_{k_m}\|_{H^1} \|h\|_{L^1} \\
 &\les& 2^{-\min(|l_m-l_{m-1}|,
|l_m-k_m|)} \|a_{k_m}\|_{H^1} \|h\|_{L^1} 
\eeaa
 as desired.

\item Assume  $l_{m-1}>k_m$ .   
In this case $P_{l_{m-1}}(a_{k_m}P_{l_m} h)=0$ unless
 $|l_{m-1}- l_m|\le 2$. Therefore we have again,
  \beaa
 \min\big(|l_m-l_{m-1}|,
|l_m-k_m|)\big)\approx 1
\eeaa
and 

\beaa
\|P_{l_{m-1}}(a_{k_m}P_{l_{m-1}} h)\|_{L^1}&\les &
\|a_{k_m}\|_{H^1}\|h\|_{L^1}\\
&\les&2^{-\min(|l_m-l_{m-1}|,
|l_m-k_m|)} \|a_{k_m}\|_{H^1} \|h\|_{L^1} 
\eeaa
\item If $l_{m-1}= k_m$,  then  $P_{l_{m-1}}(a_{k_m}P_{l_m} h)=0$
unless 
 $l_m\le k_m$. Then, using the Bernstein inequality
 $\|P_{l_m} h\|_{L^2}\les 2^{l_m} \|h\|_{L^2}$ we derive,
\beaa
\|P_{l_{m-1}}(a_{k_m} P_{l_m}h)\|_{L^1}&\les&
\|(a_{k_m}P_{l_m} h)\|_{L^1}\les \|a_{k_m}\|_{L^2}\|P_{l_m} h)\|_{L^2}\\
&\les& 2^{-k_m} \|a_{k_m}\|_{H^1} \|P_{l_m} h\|_{L^2}\\
 &\les& 2^{-k_m+l_m}
\|a_k\|_{H^1}\|h\|_{L^1}
\eeaa
Since  in this case $l_m\le k_m=l_{m-1}$ we have,
\beaa
\min\big(|l_m-l_{m-1}|,
|l_m-k_m|\big)=k_m-l_m
\eeaa
Therefore,
\beaa
 \|P_{l_{m-1}}(a_{k_m} P_{l_m}h)\|_{L^1}&\les&2^{-\min(|l_m-l_{m-1}|,
|l_m-k_m|)} \|a_{k_m}\|_{H^1} \|h\|_{L^1} 
\eeaa
as desired.
\end{enumerate}
Thus in all cases inequality 
\eqref{eq:rec-main'} is verified.
\end{proof}

\section{Proof of the main theorem}
We  need to prove the estimate
\beaa
\sup_{t\in[0,1]}\|u(t)\|_{L^1(\RRR^d)}\les  C N(g)
\eeaa
where $d=2$ and 
$$
N(g)=\|g\|_{L^1([0, 1]\times \RRR^2)}\,\,\log^{+} \big{\{}
\sup_{{\bf a}\in \ZZZ^2}
|{\bf a}|^2 \|\chi_{\bf a} g\|_{L^\infty([0,1]\times\RRR^2)}\big{\}} + 1
$$ 
 for a solution
to \eqref{eq:Tr1}
$$
\pr_t u- a(t,x) Mu = g,\qquad u(0,x)=0,
$$
where the coefficient $a$ admits the decomposition
\be{eq:decomp}
a=\pr_t b+c
\end{equation}
with $a,b$ and $c$ satisfying the conditions \eqref{eq:Tr2}, \eqref{eq1-norm}
and \eqref{eq2-norm}.

We define the iterates $u^{0}=0, u^1,\ldots u^{n}, u^{n+1}$
according to the recursive formula,
\bea
\pr_tu^{(n+1)}(t,x)=a(t_0,x) M u^{(n)}(t,x)+g(t,x),\quad 
u^{(n+1)}(0)=0.\label{eq:iterates1}
\eea
\subsection{First iterates} To illustrate our method
consider first the case of the iterate,
\beaa
u^{(2)}(t_0)=\int_0^{t_0}g(t_1) dt_1+\int_0^{t_0}a(t_1)dt_1
 M \int_0^{t_1} g(t_2) dt_2
\eeaa
Thus,
\beaa
\|\sup_{t_0\in [0,1]}u^{(2)}(t_0)\|_{L^1(\RRR^d)}&\les &
\|\sup_{t_0\in [0,1]}\int_0^{t_0}g(t_1) dt_1\|_{L^1}
+\|\sup_{t_0\in
[0,1]} I(t_0)\|_{L^1}\\
I(t_0)&=&\int_0^{t_0}a(t_1)dt_1
 M \int_0^{t_1} g(t_2) dt_2
\eeaa
The first term is trivial.
To estimate the second term we need to make use of the decomposition \eqref{eq:decomp}.
Thus,
\beaa
I(t_0)&=&I_b(t_0)+I_c(t_0) \\
I_c(t_0)&=&\int_0^{t_0}c(t_1)dt_1
  \int_0^{t_1}M g(t_2) dt_2\\
I_b(t_0)&=&\int_0^{t_0}\pr_{t_1}b(t_1)dt_1
  \int_0^{t_1} Mg(t_2) dt_2\\
&=&b(t_0)\int_0^{t_0}M g(t_2) dt_2-\int_0^{t_0}b(t_1)Mg(t_1) dt_1\\
&:=&
I_{b,1}(t_0)+I_{b,2}(t_0)
\eeaa
To estimate $I_c$ we use the fact that, for $d=2$, the Besove space
$B_{2,1}^1(\RRR^d)$ embedds in $L^\infty(\RRR^d)$ and the estimate,
\beaa
\|Mg(t)\|_{L^1(\RRR^d)}\les \|g(t)\|_{L^1(\RRR^d)}\log^{+}
\|g(t)\|_{L^\infty(\RRR^d)} +1\les N(g(t))
\eeaa
 Thus,
\beaa
\|\sup_{t_0\in [0,1]}I_c(t_0)\|_{L^1}&\les &
\int_0^{1}\|c(t_1)\|_{L^\infty}dt_1
  \int_0^{t_1}\|M g(t_2)\|_{L^1(\RRR^d)} dt_2\\
&\les&\int_0^{1}\|c(t_1)\|_{B^1_{2,1}(\RRR^d)}dt_1\int_0^{t_1}N(g)(t_2) dt_2\\
&\les&\|c\|_3 N(g)
\eeaa
On the other hand, decomposing $b=b_0+\sum_{k\in \ZZZ_+} b_k$,
\beaa
\|\sup_{t_0\in [0,1]}I_{b,1}(t_0)\|_{L^1(\RRR^d)}&\les& 
\|\sup_{t_0\in [0,1]} b(t_0)\|_{L^\infty(\RRR^d)}\int_0^{t_0}\|M
g(t_2)\|_{L^1(\RRR^d)} dt_2\\
&\les& N(g)\|\sup_{t_0\in [0,1]} b(t_0)\|_{L^\infty(\RRR^d)} \\
&\les& N(g)\sum_{k\in\ZZZ_+\cup \{0\}}\|\sup_{t_0\in [0,1]} b_k(t_0)\|_{L^\infty(\RRR^d)}
\eeaa 
We now appeal to the following straightforward  lemma,
\begin{lemma} The following calculaus inequality holds true (see \eqref{eq1-norm})
for $k\ge 0$,
\beaa
\sup_{t\in [0,1]}\|b_k(t)\|_{H^1(\RRR^d)}\les \|\pr_t
b_k\|_{L_t^2H^1}^{1/2}\|b_k\|_{L_t^2H^1}^{1/2}\les 2^{-k/2}
\|b_k\|_{2}
\eeaa
Also, 
\beaa 
\|\sup_{t\in [0,1]} b_k(t)\|_{L^\infty(\RRR^d)}\les 
\|\pr_t b_k\|_{L_t^2H^1}^{1/2}\|b_k\|_{L_t^2H^1}^{1/2}\les 2^{-k/2}
\|b_k\|_{2}
\eeaa

\label{Le:2.3}
\end{lemma}
In view of the Lemma we deduce,
\beaa
\|\sup_{t_0\in [0,1]}I_{b,1}(t_0)\|_{L^1(\RRR^d)}&\les& 
N(g)\sum_{k\in\ZZZ_+\cup\{0\}}\|b_k\|_{L_t^2H^1}\\
&\les&N(g)\sum_{k\in\ZZZ_+\cup\{0\}}2^{-k/2} \|b_k\|_2
\les N(g) \|b\|_2
\eeaa
Similarly,
\beaa
\|\sup_{t_0\in [0,1]}I_{b,2}(t_0)\|_{L^1(\RRR^d)}&\les&
\|\int_0^{1}b(t_1)Mg(t_1) dt_1\|_{L^1(\RRR^d)}\\
&\les&N(g)\sup_{t_1\in[0,1]}\|b(t_1)\|_{L^\infty}\\
&\les& N(g) \|b\|_2
\eeaa
Therefore,
\beaa
\|\sup_{t_0\in [0,1]}u^{(2)}(t_0)\|_{L^1(\RRR^d)}\les
N(g)\big(\|b\|_2+\|c\|_3\big)
\eeaa
\begin{remark}
Observe that there is room of a $1/2$ derivative 
in the estimates for $I_b$. This  room will play an important role
for treating the general iterates $u^{(n+1)}$.
\end{remark}
Consider now the  more dificult  case of the  iterate $u^{(3)}$,
\beaa
u^{(3)}&=&\int_0^{t_0}g(t_1)
dt_1+\int_{0}^{t_0}a(t_1) M u^{(2)}(t_1)dt_1\\
&=&\int_0^{t_0} g(t_1)dt_1 +
\int_0^{t_0} a(t_1)dt_1M\big(\int_0^{t_1}g(t_2) dt_2\big)\\
&+&\int_0^{t_0}\int_0^{t_1}\int_0^{t_2}a(t_1)M a(t_2)
Mg(t_{3})dt_1 dt_2 dt_{3}
\eeaa
We concentrate our attention on the last term,
\beaa
I(t_0)=\int_0^{t_0}\int_0^{t_1}\int_0^{t_2}a(t_1)M a(t_2)
Mg(t_{3})dt_1 dt_2 dt_{3}
\eeaa
As we decompose each $a(t_i)=\pr_t b(t_i) +c(t_i)$ with $i=1,2$ we notice 
that we can only integrate by parts only one of the potentially two terms containing 
$\pr_t b(t_i)$. We need to make that choice judiciously, based on the relative strength
of the terms.  
We begin by decomposing
$a(t_1), a(t_2)$ into their Littlewood-Paley pieces and write,
\beaa
I(t_0)&=&\int_0^{t_0}\int_0^{t_1}\int_0^{t_2}\sum_{k_1,k_2\in\ZZZ_+\cup\{0\}}a_{k_1}(t_1)M
a_{k_2}(t_2) Mg(t_{3})dt_1 dt_2 dt_{3}\\
&=&\int_0^{t_0}\int_0^{t_1}\int_0^{t_2}\sum_{0\le k_1<k_2}+
\int_0^{t_0}\int_0^{t_1}\int_0^{t_2}\sum_{0\le k_1=k_2}
+\int_0^{t_0}\int_0^{t_1}\int_0^{t_2}\sum_{k_1>k_2\ge 0}
\eeaa
In what follows we will tacitly assume that all the integer indices $k_i$ take values in 
the set of non-negative integers and will not write this constraint explicitly. 
Consider the last term,
\beaa
J(t_0)=\int_0^{t_0}\int_0^{t_1}\int_0^{t_2}\sum_{k_1>k_2}a_{k_1}(t_1)M
a_{k_2}(t_2) Mg(t_{3})dt_1 dt_2 dt_{3}
\eeaa
We further decompose,
\beaa
a_{k_1}(t_1)=\pr_{t}b_{k_1}(t_1)+c_{k_1}(t_1)
\eeaa
and concentrate on the term,
\beaa
J_b(t_0)&=&\int_0^{t_0}\int_0^{t_1}\int_0^{t_2}\sum_{k_1>k_2}\pr_{t_1}
b_{k_1}(t_1)M a_{k_2}(t_2) Mg(t_{3})dt_1 dt_2 dt_{3}\\
&=&\sum_{k_1>k_2}b_{k_1}(t_0)\int_0^{t_0}\int_0^{t_2}
M a_{k_2}(t_2) Mg(t_{3}) dt_2 dt_{3}\\
&-&\sum_{k_1>k_2}\int_0^{t_0}\int_0^{t_1} b_{k_1}(t_1)M
a_{k_2}(t_1) Mg(t_{3})dt_1  dt_{3}
\eeaa
Let,
\beaa
J_{b1}(t_0)=\sum_{k_1>k_2}b_{k_1}(t_0)\int_0^{t_0}\int_0^{t_2}
M a_{k_2}(t_2) Mg(t_{3}) dt_2 dt_{3}
\eeaa
and estimate
\beaa
\|J_{b1}(t_0)\|_{L^1}&\les&\sum_{k_1>k_2}
\|b_{k_1}(t_0)\|_{L^\infty}\int_0^{t_0}\int_0^{t_2}
\|M a_{k_2}(t_2) Mg(t_{3})\|_{L^1} dt_2 dt_{3}
\eeaa
Using  Lemma \ref{le:main} we have,
\beaa
\|M a_{k_2}(t_2) Mg(t_{3})\|_{L^1}\les \|a_{k_2}(t_2)\|_{H^1}
N(g)(t_3)
\eeaa
 Also, according to Lemma \ref{Le:2.3}, 
using the norm $\|\,\|_2$ introduced in \eqref{eq2-norm},
\beaa
\|b_{k_1}(t_0)\|_{L^\infty}\les 2^{-k_1/2}\|b_{k_1}\|_2
\eeaa
Hence,
\beaa
\|J_{b1}(t_0)\|_{L^1}&\les&\sum_{k_1>k_2\ge 0} 2^{-k_1/2}
\|b_{k_1}\|_{2}\int_0^{t_0}\|a_{k_2}(t_2)\|_{H^1} dt_2\int_0^{t_2}
 N(g)(t_{3}) dt_{3} dt_3\\
&\les& N(g)\sum_{k_1>k_2\ge 0} 2^{-k_1/2}\|b_{k_1}\|_{2}\|a_{k_2}\|_{1}
\les N(g)\|b\|_2\|a\|_1
\eeaa
The term $J_{b2}=\sum_{k_1>k_2}\int_0^{t_0}\int_0^{t_2} b_{k_1}(t_1)M
a_{k_2}(t_1) Mg(t_{3})dt_1  dt_{3}$
can be treated in exactly the same fashion. Thus,
\bea
\|J_{b}(t_0)\|_{L^1}&\les&N(g)\|b\|_2\|a\|_1
\eea
Consider now the term,
\beaa
J_c(t_0)&=&\int_0^{t_0}\int_0^{t_1}\int_0^{t_2}\sum_{k_1>k_2}
c_{k_1}(t_1)M a_{k_2}(t_2) Mg(t_{3})dt_1 dt_2 dt_{3}
\eeaa
We further decompose
\beaa
a_{k_2}(t_2)=\pr_{t}b_{k_2}(t_2)+c_{k_2}(t_2)
\eeaa
We show how to treat the term,
\beaa
J_c(t_0)&=&\int_0^{t_0}\int_0^{t_1}\int_0^{t_2}\sum_{k_1>k_2}
c_{k_1}(t_1)M\pr_{t}b_{k_2}(t_2) Mg(t_{3})dt_1 dt_2 dt_{3}\\
&=&\sum_{k_1>k_2}\int_0^{t_0}\int_0^{t_1}
c_{k_1}(t_1)Mb_{k_2}(t_1) Mg(t_{3})dt_1 dt_{3}\\
&-&\sum_{k_1>k_2}\int_0^{t_0}\int_0^{t_1}
c_{k_1}(t_1)Mb_{k_2}(t_2) Mg(t_{2})dt_1 dt_{2}
\eeaa
Hence, using first Lemma  \ref{le:main} followed by Lemma \ref{Le:2.3},
\beaa
\|J_c(t_0)\|_{L^1}&\les&
\sum_{k_1>k_2}\int_0^{t_0}\int_0^{t_1}
\|c_{k_1}(t_1)Mb_{k_2}(t_1) Mg(t_{3})\|_{L^1}dt_1 dt_{3}\\
&+&\sum_{k_1>k_2}\int_0^{t_0}\int_0^{t_1}
\|c_{k_1}(t_1)Mb_{k_2}(t_2) Mg(t_{2})\|_{L^1}dt_1 dt_{2}\\
&\les&\sum_{k_1>k_2}\int_0^{t_0}\int_0^{t_1}
\|c_{k_1}(t_1)\|_{H^1}\|b_{k_2}(t_1)\|_{H^1} N(g)(t_{3})dt_1
dt_{3}\\
&+&\sum_{k_1>k_2}\int_0^{t_0}\int_0^{t_1}
\|c_{k_1}(t_1)\|_{H^1}\|b_{k_2}(t_2)\|_{H^1} N(g)(t_{2})dt_1
dt_{2}\\
&\les&\sum_{k_1>k_2}\sup_{t\in
[0,1]}\|b_{k_2}(t)\|_{H^1}\int_0^{t_0}\int_0^{t_1}
\|c_{k_1}(t_1)\|_{H^1} N(g)(t_{2})dt_1dt_{2}\\
&\les&N(g)
\sum_{k_1>k_2\ge 0}2^{-k_2/2}\|b_{k_2}\|_{2}\|c_{k_1}\|_{L^1H^1}\les
N(g)\|b\|_2\sum_{k_1}\|c_{k_1}\|_{L^1H^1}\\
&\les&N(g)\|b\|_2\|c\|_3
\eeaa

\subsection{General case} Treatment of the general case will follow
the scheme laid down for the third iterate $u^{(3)}$. Additional challenge however
is presented in controlling constants in the estimates, which may  
grow uncontrollably with respect to  the order  of the iterates. Recalling
\eqref{eq:iterates1} we write,
\beaa
&&u^{(n+1)}(t)=\int_0^{t} g(t_1)dt_1 +
\int_0^{t} a(t_1)dt_1\int_0^{t_1}Mg(t_2) dt_2+\ldots \\
&+&\int_0^{t}\int_0^{t_1}\ldots\int_0^{t_{n}}a(t_1)M a(t_2)
M\ldots a(t_n)Mg(t_{n+1})dt_1 dt_2\ldots dt_{n+1}
\eeaa
To simplify notations introduce the simplex $\De_n(t)$ defined by,
\beaa
t\ge t_1\ge t_2\ldots\ge t_n\ge t_{n+1}\ge 0
\eeaa
and write,
\bea
u^{(n+1)}(t)&=&u^{(n)}(t)+J_n(t)
\eea
 where,
\beaa
J_n(t)&=&\int_{\De_n(t_0)} 
a(t_1)M a(t_2)  M\ldots a(t_m)M g(t_{n+1})\\ &:=&
\int\ldots\int_{\De_n(t_0)} dt_1\ldots
dt_{n+1}\,\,
a(t_1)M a(t_2)  M\ldots a(t_m)M g(t_{n+1})\nn
\eeaa
To prove \eqref{eq:Thm1} it will suffice to show that
\be{eq:n}
\sup_{t\in [0,1]}\|J_n(t)\|_{L^1(\RRR^d)}\les C^n \De^n N(g)
\end{equation}

 We decompose each  
$a(t_i)$  in the expression for $J_n$  into its Littlewood-Paley components according to,
\beaa
a(t_i)=\sum_{k\in \ZZZ_+\cup\{0\}} P_k a(t_i)=a_0(t_i)+\sum_{k_i\in \ZZZ_+}a_{k_i}(t_i)
\eeaa
Thus, writing $\k=(k_1,\ldots k_n)\in (\ZZZ_+\cup\{0\})^n$
\bea
J_n(t)=J(t)&=&\sum_{\k\in (\ZZZ_+\cup\{0\})^n} \int_{\De_n(t)}
a(t_1)_{k_1}M\ldots a_{k_n}(t_n)M g(t_{n+1})\label{eq:Jn}
\eea
For each $1\le j\le n$ we define,
 \bea
[k_j]=\{(k_1,k_2,\ldots k_n)\in (\ZZZ_+\cup\{0\})^n\,\,|\,\,k_i\le k_j\quad 
\forall i\}
\eea
to be the set on n-tuples $(k_1,...,k_n)$ with the property that for each $i=1,..,n$ 
$k_i\le k_j$. In what follows we will tacitly assume that all indices $k_i$ take values in 
the set of non-negative integers and will not write this constraint explicitly.
Let,
\bea
J_n^j(t)=J^j(t)&=&\sum_{\k\in [k_j]} \int_{\De_n(t)}
a_{k_1}(t_1) M\ldots a_{k_n}(t_n)M g(t_{n+1})\label{eq:Jn}
\eea
Clearly,
\beaa
\|J_n(t)\|_{L^1(\RRR^d)}\les
 \sum_{j=1}^n\|J^j_n(t)\|_{L^1(\RRR^d)}
\eeaa
We now fix $j$ and decompose in view of \eqref{eq:decomp},
 \bea
a_{k_j}(t_j)=\pr_{t} b_{k_j}(t_j)+c_{k_j}(t_j)
\eea 
Thus,
\bea
J^j(t)&=&J^j_b(t)+J^j_c(t)=\sum_{\k\in [k_j]} J^j_{b,\k}(t)+
\sum_{\k\in [k_j]}J^j_{c,\k}(t)\\
 J_{b,\k}^j(t)&=&
 \int_{\De_n(t)}a_{k_1}(t_1) M\ldots\pr_{t}b_{k_j}(t_j)M\ldots
a_{k_n}(t_n)M g(t_{n+1})dt_1\ldots dt_{n+1}
\nn\\
 J_{c,\k}^j(t)&=&
 \int_{\De_n(t)}a_{k_1}(t_1)M\ldots c_{k_j}(t_j)M\ldots
a_{k_n}(t_n)M g(t_{n+1})dt_1\ldots dt_{n+1}\nn
\eea
with the summation convention,
 \beaa
\sum_{\k\in [k_j]}=\sum_{k_j\in\ZZZ}\,\,\, \sum_{\k'\le
k_j},\qquad\quad  \k'=(k_1,\ldots \widehat{k_j}\ldots k_n).
\eeaa  
We  first estimate\footnote{For simplicity, since $j$ is kept fix
 we drop the $j$ upper index below}
$J_b=J_b^j$.   Integrating by parts,
\beaa
 J_{b,\k}(t)&=&
 \int_{\De_{n-1}(t)}\ldots a_{k_{j-1}}(t_{j-1})M b_{k_j}(t_{j-1})M
a_{k_{j+1}}(t_{j+1}) \ldots
M g(t_{n+1})dt_1\ldots\widehat{dt_j}\ldots dt_{n+1}\\
&-&\int_{\De_{n-1}(t)}\ldots a_{k_{j-1}}(t_{j-1})M b_{k_j}(t_{j+1})M a_{k_{j+1}}(t_{j+1})
\ldots
M g(t_{n+1})dt_1\ldots\widehat{dt_{j}}\ldots dt_{n+1}\\
&=&J^{-}_{b,\k}(t)+J^{+}_{b,\k}(t)
\eeaa
Now,  with the help of Lemma  \ref{le:main}, we proceed as 
in the previous subsection,
\beaa
\|J^{-}_{b,\k}(t)\|_{L^1}&\les& C^n\sup_{t}\|b_{k_j}(t)\|_{H^1}
\int_{\De_{n-1}(t)}A_\k(t_1,\ldots
\widehat{t_{j}}\ldots t_n )  N(g)(t_{n+1})dt_1\ldots\widehat{dt_j}\ldots
dt_{n+1}
\eeaa
where,
\beaa
A_{\k,j}(\ldots
\widehat{t_{j}}\ldots  )&=&\|a_{k_1}(t_1)\|_{H^1}\ldots 
\widehat{\|a_{k_j}(t_j)\|_{H^1}}\ldots \|a_{k_n}(t_n)\|_{H^1}
\eeaa
Henceforth, with the help of Lemma \ref{Le:2.3},
\beaa
\|J^{-}_{b,\k}(t)\|_{L^1}&\les&  C^nN(g)2^{-k_j/2}\|b_{k_j}\|_2\,\, 
|\De_{n-2}(t)|^{1/2}\big(\int_{\De_{n-2}(t)}A_\k(\ldots
\widehat{t_{j}}\ldots  )^2dt_1\ldots\widehat{dt_j}\ldots
dt_{n}\big)^{1/2}
\eeaa
where $|\De_{n-2}(t)|$ is the volume of the $n-2$ dimensional simplex\footnote{In our notations 
it corresponds to an actual $(n-1)$-dimensional simplex.}.
Consequently,
\beaa
\|J^{-}_{b,\k}(t)\|_{L^1}&\les&C^n((n-1)!)^{-1/2}
N(g)2^{-k_j/2}\|b_{k_j}\|_2
\|a_{k_1}\|_1\ldots \widehat{\|a_{k_j}\|_1}\ldots\|a_{k_n}\|_1
\eeaa
and, by triangle inequality and  then Cauchy-Schwartz,
\beaa
\|\sum_{\k\in [k_j]}J^{-}_{b,\k}(t)\|_{L^1}&\les&C^n((n-1)!)^{-1/2}
N(g)\sum_{\k\in [k_j]}2^{-k_j/2}\|b_{k_j}\|_2
\|a_{k_1}\|_1\ldots \widehat{\|a_{k_j}\|_1}\ldots\|a_{k_n}\|_1\\
&\les&C^n ((n-1)!)^{-1/2}
N(g)(\sum_{\k\in [k_j]}2^{-k_j})^{1/2}\big(\sum_{\k\in [k_j]}\|b_{k_j}\|_2^2
\|a_{k_1}\|_1^2\ldots\|a_{k_n}\|_1^2\big)^{1/2}\\
&\les&C^n(\frac{n!}{(n-1)!})^{1/2}N(g)\|b\|_2
\|a\|_1^{n-1}\\
&\les& n^{\frac 12} C^nN(g)\|b\|_2
\|a\|_1^{n-1}
\eeaa
Proceeding  exactly in the same way we derive,
\beaa
\|\sum_{\k\in
[k_j]}J^{+}_{b,\k}(t)\|_{L^1}&\les&nC^nN(g)\|b\|_2
\|a\|_1^{n-1}
\eeaa
Therefore, recalling that $J_b(t)=\sum_{\k\in [k_j]} J_{b,\k}(t)$,
\bea
\|J_b^j(t)\|_{L^1(\RRR^d)}&\les&nC^nN(g)\|b\|_2
\|a\|_1^{n-1}
\eea

To estimate $J_c^j(t)=\sum_{\k\in[ k_j]} J_{c,\k}(t)$ we have to do
a further decomposition. We define,
\bea
[ k_j,k_l]=\{(k_1,k_2,\ldots k_n)\in(\ZZZ_+\cup\{0\})^n\,\,|\,\,k_i\le
k_l\le k_j\quad 
\forall i\neq l,j\}
\eea
For fixed  $j$  we have precisely $n-1$ such regions
 covering $[k_j]$. Fix $l\neq j$  and consider,
\bea
J_c^{jl}(t)=\sum_{\k\in[ k_j,k_l]} J_{c,\k}^{jl}(t)
\eea
Clearly,
\bea
\|J_c^j(t)\|_{L^1(\RRR^d)}&\les&\sum_{l\neq
j}\|J_{c,\k}^{jl}(t)\|_{L^1(\RRR^d)}
\eea
In view of \eqref{eq:decomp} we
decompose,
 \bea
a_{k_l}(t_l)=\pr_{t} b_{k_l}(t_l)+c_{k_l}(t_l)
\eea
Thus, dropping the upper indices $j,l$,
\bea
J_c(t)&=&J_{cb}(t)+J_{cc}(t)=\sum_{\k \in [k_j,k_l]} J_{cb,\k}(t)+
\sum_{\k\in [k_j,k_l]}J_{cc,\k}(t)\\
 J_{cb,\k}(t)&=&
 \int_{\De_n(t)}a_{k_1}(t_1)M\ldots
c_{k_j}(t_j)M\ldots\pr_{t}b_{k_l}(t_l)M\ldots a_{k_n}(t_n)M
g(t_{n+1})dt_1\ldots dt_{n+1}
\nn\\
 J_{cc,\k}(t)&=&
 \int_{\De_n(t)}a(t_1)_{k_1}M\ldots
c_{k_j}(t_j)M\ldots c_{k_l}(t_l)\ldots a_{k_n}(t_n)M
g(t_{n+1})dt_1\ldots dt_{n+1}\nn
\eea 
  Integrating by parts, and droping the operators $M$ for a moment,
\beaa
 J_{cb,\k}(t)&=&
 \int_{\De_{n-1}(t)}\ldots c_{k_j}(t_j)\ldots
a_{k_{l-1}}(t_{l-1})b_{k_l}(t_{l-1}) a_{k_{l+1}}(t_{l+1}) \ldots
g(t_{n+1})dt_1\ldots\widehat{dt_l}\ldots dt_{n+1}\\
&-&\int_{\De_{n-1}(t)}\ldots c_{k_j}(t_j)\ldots a_{k_{l-1}}(t_{l-1}) b_{k_l}(t_{l+1})
a_{k_{l+1}}(t_{l+1})  a_{k_{l+2}}(t_{l+2})\ldots
 g(t_{n+1})dt_1\ldots\widehat{dt_{l}}\ldots dt_{n+1}\\
&=&J^{-}_{cb,\k}(t)+J^{+}_{cb,\k}(t)
\eeaa
Using Lemma  \ref{le:main} as before,
\beaa
\|J^{\pm}_{cb,\k}(t)\|_{L^1}&\les& C^n\sup_{t}\|b_{k_l}(t)\|_{H^1}
\int_{\De_{n-1}(t)}B_{\k}(t_1,\ldots
\widehat{t_{l}}\ldots t_n )  N(g)(t_{n+1})dt_1\ldots\widehat{dt_l}\ldots
dt_{n+1}
\eeaa
where,
\beaa
B_{\k}(\ldots
\widehat{t_{l}}\ldots  )&=&\|a_{k_1}(t_1)\|_{H^1}\ldots
\|c_{k_j}(t_j)\|_{H^1}\ldots
\widehat{\|a_{k_l}(t_l)\|_{H^1}}\ldots \|a_{k_n}(t_n)\|_{H^1}
\eeaa
Therefore, exactly as before with the help of Lemma \ref{Le:2.3},
\beaa
\|J^{\pm}_{cb,\k}(t)\|_{L^1}&\les& C^n N(g)2^{-k_l/2}\|b_{k_l}\|_2\,\,
P_{\k,n-2}(t)\\
P_{\k,n-2}(t)&=&\int_{\De_{n-2}(t)}B_\k(\ldots
\widehat{t_{l}}\ldots  )dt_1\ldots\widehat{dt_l}\ldots
dt_{n}
\eeaa
Observe that,
\beaa
P_{\k,n-2}(t)
&\le&\int_{\De_{n-2}(t)}
\|a_{k_1}(t_1)\|_{H^1}\ldots
\|c_{k_j}(t_j)\|_{H^1}\ldots\widehat{\|a_{k_l}(t_l)\|_{H^1}}\dots
 \|a_{k_n}(t_{n})\|_{H^1}dt_1\ldots\widehat{dt_l}\ldots dt_{n}
\eeaa

Thus,
\beaa
&&\|\sum_{\k\in[k_j,k_l]}J^{\pm}_{cb,\k}(t)\|_{L^1}\les C^n
((n-2)!)^{-1/2}N(g) Q
\eeaa
with,
\beaa
Q&=&\sum_{k_l\le
k_j}2^{-k_l/2}\|b_{k_l}\|_2\|c_{k_j}\|_{3}\sum_{\k''\le
k_l}\|a_{k_1}\|_1\ldots 
\widehat{\|a_{k_j}\|_1}\ldots\widehat{\|a_{k_l}\|_1}\ldots
\|a_{k_n}\|_1
\eeaa
with $k''=(k_1,k_2\ldots, \widehat{k_j},\ldots,\widehat{k_l}\ldots
k_n)$. Therefore, by Cauchy-Schwartz,
\beaa
Q&\les&\sum_{k_l\le
k_j}2^{-k_l/2}k_l^{(n-2)/2} \|b_{k_l}\|_2\|c_{k_j}\|_{3}
\big(\sum_{\k''\le k_l}\|a_{k_1}\|_1^2\ldots 
\|a_{k_n}\|_1^2\big)^{1/2}\\
&&\les\|a\|_{1}^{n-2}\sum_{k_j\in\ZZZ}\|c_{k_j}\|_{3}
\sum_{k_l\le
k_j}2^{-k_l/2}k_l^{(n-2)/2} \|b_{k_l}\|_2\\
&&\les\|a\|_{1}^{n-2}\|b\|_2\sum_{k_j\in\ZZZ}
\|c_{k_j}\|_{3}\big(\sum_{k_l=0}^{k_j}
2^{-k_l}k_l^{(n-2)}\big)^{1/2}\\
&&\les ((n-1)!)^{1/2}\|a\|_{1}^{n-2}\|b\|_2\|c\|_3
\eeaa
Consequently,
\beaa
\|\sum_{\k\le k_l\le k_j}J^{\pm}_{cb,\k}(t)\|_{L^1}&\les& C^n
\big(\frac{(n-1)!}{(n-2)!}\big)^{1/2}N(g) \|a\|_{1}^{n-2}
\|b\|_2\|c\|_3\\
&\les&n^{\frac 12} C^n N(g) \|a\|_{1}^{n-2}
\|b\|_2\|c\|_3
\eeaa
Therefore,
\bea
\sup_{t\in[0,1]}\|J_{cb}^{jl}(t)\|_{L^1(\RRR^d)}\les
n^{\frac 12} C^n N(g) \|a\|_{1}^{n-2}\|b\|_2\|c\|_3
\eea
To treat the term $J_{cc,\k}(t)$ we decompose once more. 
Continuing in the same manner  after $m$ steps we arrive
at the integral,
\bea
J_{c_1\ldots c_{m-1}}^{j_1j_2\ldots j_{m-1}}(t)&=&\sum_{
[k_{j_1}\ldots k_{j_{m-1}}]}
\int_{\De_n(t)}\ldots
\eea
with the integrand containing $c_1=c_{k_{j_1}},
c_2=c_{k_{{j_2}}}\ldots c_{m-1}=c_{k_{j_{m-1}}}$ and
\beaa
[k_{j_1},\ldots k_{j_{m-1}}]=\{(k_1\ldots
k_n)\in\ZZZ^n\,\,|\,\,k_i\le k_{j_m}\le\ldots\le  k_{j_1}\quad 
\forall i\neq j_1,j_2\ldots j_{m-1}\}
\eeaa
Clearly $[k_{j_1},\ldots k_{j_{m-1}}]$ can be covered by
precisely 
$n-m+1$ regions of the form $[k_{j_1},\ldots k_{j_m}]$. We have,
\bea
J_{c_1\ldots c_{m-1}}^{j_1j_2\ldots j_{m-1}}(t)&=&\sum_{j_m} 
J_{c_1\ldots c_{m-1}}^{j_1j_2\ldots j_{m}}(t),\qquad k_{j_m}\le
k_{j_{m-1}}\\ J_{c_1\ldots c_{m-1}}^{j_1j_2\ldots j_{m}}(t)&=&\sum_{
[k_{j_1}\ldots k_{j_{m}}]}
\int_{\De_n(t)}\ldots
\eea
In view of \eqref{eq:decomp} we
decompose,
 \bea
a_{k_{j_m}}(t_{j_m})=\pr_{t} b_{k_{j_m}}(t_{j_m})+c_{k_{j_m}}(t_{j_m})
\eea
and, respectively,
\beaa
J_{c_1\ldots c_{m-1}}^{j_1j_2\ldots j_{m}}(t)=\sum_{\k \in[
k_{j_1}\ldots k_{j_m}]} J_{c_1\ldots c_{m-1}b_m,\k}^{j_1j_2\ldots
j_{m}}(t)+
\sum_{\k\in[ k_{j_1}\ldots
k_{j_m}]} J_{c_1\ldots c_{m},\k}^{j_1j_2\ldots j_{m}}(t)
\eeaa
where $b_m=b_{k_{j_m}}$, $c_m=c_{k_{j_m}}$
Proceeding exactly as before, integrating by parts and using Lemma  \ref{le:main},
we write,
\beaa
\|J_{c_1\ldots c_{m-1}b_m,\k}^{j_1j_2\ldots
j_{m}}(t)\|_{L^1}&\les& C^n\sup_{t}\|b_{k_{j_m}}(t)\|_{H^1}
\int_{\De_{n-1}(t)}B_\k(t_1,\ldots
\widehat{t}_{j_m}\ldots t_n )  N(g)(t_{n+1})
\eeaa
where,
\beaa
B_{\k}(\ldots
\widehat{t}_{j_m}\ldots  )&=&\|c_{k_{j_1}}(t_{j_1})\|_{H^1}\ldots
\|c_{k_{j_{m-1}}}(t_{j_{m-1}})\|_{H^1}\\
&\cdot&
\|a_{k_{j_{m+1}}}(t_{j_{m+1}})\|_{H^1}\ldots  \|a_{k_{j_n}}(t_{j_n})\|_{H^1}
\eeaa
Therefore,
\beaa
\|J_{c\ldots cb,\k}^{j_1j_2\ldots
j_{m}}(t)\|_{L^1}&\les& C^n N(g)2^{-k_{j_m}/2}\|b_{k_{j_m}}\|_2\,\,
P_{\k,n-2}(t)\\
P_{\k,n-2}(t)&=&\int_{\De_{n-2}(t)}B_{\k}(\ldots
\widehat{t}_{j_m}\ldots  )
\eeaa
where $k_{j_{m+1}},\ldots k_{{j_n}}$ are the  labels for all
other frequencies different  from  $k_{j_1},\ldots k_{j_{m-1}}$.

To estimate $P_{\k,n-2}(t)$ we make use of
the following  obvious lemma.
\begin{lemma}\label{le:combin} Let $f_1,f_2,\ldots f_{n}$  be an
ordered sequence  of
$n$
 positive, integrable, functions defined on the interval
 $[0,1]\subset\RRR$ among
which
$m$, say $f_{i_1}, i=1,\ldots m$ are in $L^1$ and
$n-m$, say $f_{j_1},\ldots f_{j_{n-m}}$ are in $L^2$. 
Then,
\beaa
\int_{\De_{n-2}(t)}f_1(t_1)\ldots f_{n}(t_{n})dt_1\ldots dt_{n
}&\les&
\big(\frac{1}{(n-m)!}\big)^{1/2}\|f_{i_1}\|_{L^1}\ldots 
\|f_{i_m}\|_{L^1}\\&\c&\|f_{j_1}\|_{L^1}\ldots 
\|f_{j_{n-m}}\|_{L^1}
\eeaa
\end{lemma}
According to Lemma \ref{le:combin} we have,
\beaa
P_{\k,n-2}(t)&\les&\big(\frac{1}{(n-m-1)!}\big)^{1/2}
\|c_{k_{j_1}}\|_{L^1
H^1}\ldots 
\|c_{k_{j_{m-1}}}\|_{L^1 H^1}\c\|a_{k_{j_{m+1}}}\|_{1}\ldots 
\|a_{k_{j_{n}}}\|_{1}
\eeaa
Observe that,
\beaa
\sum_{k''\le k_{j_m}}\|a_{k_{j_{m+1}}}\|_{1}\ldots 
\|a_{k_{j_{n}}}\|_{1}&\les& \big(k_{j_m}\big)^{(n-1-m)/2} (\sum_{k''\le
k_{j_m}}\|a_{k_{j_{m+1}}}\|_{1}^2\ldots 
\|a_{k_{j_{n}}}\|_{1}^2)^{1/2}\\
&\les&\big(k_{j_m}\big)^{(n-1-m)/2} \|a\|_{1}^{m-n}
\eeaa
where $k''=(k_{j_{m+1}},\ldots k_{{j_n}})$.
 Observe also that,
\bea
\sum_{k_{j_1}\le k_{j_2}\ldots \le k_{j_{m-1}}}
\|c_{k_{j_1}}\|_{L^1
H^1}\ldots 
\|c_{k_{j_{m-1}}}\|_{L^1 H^1}\les
\frac{1}{(m-1)!}\|c\|_{3}^{m-1}
\eea
Indeed this follows by symmetry in view of the fact that,
\beaa
\sum_{k_{j_1},\ldots, k_{j_m}}\|c_{k_{j_1}}\|_{L^1
H^1}\ldots 
\|c_{k_{j_{m-1}}}\|_{L^1 H^1}\les
\|c\|_{3}^{m-1}
\eeaa
Finally, by Cauchy-Schwartz,
\beaa
\sum_{
k_{j_m}\in\ZZZ}2^{-k_{j_m}/2}\big(k_{j_m}\big)^{(n-1-m)/2}
\|b_{k_{j_m}}\|_2\les ((n-m)!)^{1/2}\|b\|_2
\eeaa
Hence,
\beaa
\sum_{[ k_{j_1}\ldots
k_{j_m}]}\| J_{c\ldots cb,\k}^{j_1j_2\ldots j_{m}}(t)\|_{L^1}
&\les&  C^n \frac{1}{(m-1)!}
\big(\frac{(n-m)!}{(n-m-1)!}\big)^{1/2}N(g)\|b\|_{2}\|a\|_1^{n-m}
\|c\|^{m-1}_{3}
\eeaa
In other words,
\bea
\sum_{[ k_{j_1}\ldots
k_{j_m}]}\| J_{c\ldots cb,\k}^{j_1j_2\ldots j_{m}}(t)\|_{L^1}
&\les&n^{\frac 12} C^n \frac{1}{(m-1)!}\De_0^n
\eea

We are ready to estimate $J_n(t)=J(t)$ in formula
\eqref{eq:Jn}. We have, 
\beaa
\|J(t)\|_{L^1)}\les
 \sum_{j_{1}=1}^n\|J^{j_1}(t)\|_{L^1}
\eeaa
and, 
\beaa
\|J^{j_1}(t)\|_{L^1}&\les&\|J^{j_1}_{b_1}(t)\|_{L^1}+
\|J^{j_1}_{c_1}(t)\|_{L^1}\\
&\les&n^{\frac 12} C^n\De_0^n+\|J^{j_1}_{c_1}(t)\|_{L^1}
\eeaa
Hence,
\beaa
\|J(t)\|_{L^1}&\les&n^{\frac 32}C^n\De_0^n+
 \sum_{j_{1}=1}^n\|J^{j_1}_{c_1}(t)\|_{L^1}
\eeaa
On the other hand, for each $j_1$,
\beaa
\|J^{j_1}_{c_1}(t)\|_{L^1}\les
\sum_{j_{2}\neq j_1}^n\|J^{j_1j_2}_{c_1}(t)\|_{L^1}
\eeaa
and,
\beaa
\|J^{j_1j_2}_{c_1}(t)\|_{L^1}&\les &
\|J^{j_1j_2}_{c_1b_2}(t)\|_{L^1}+
\|J^{j_1j_2}_{c_1c_2}(t)\|_{L^1}\\
&\les&n^{\frac 12} \frac {C^n\De_0^n}{1!}+\|J^{j_1j_2}_{c_1c_2}(t)\|_{L^1}
\eeaa
Therefore,
\beaa
\|J(t)\|_{L^1(\RRR^d)}&\les&n^{\frac 12}  n C^n\De_0^n+n^{\frac 12} \frac {n(n-1)}{1!} 
C^n\De_0^n
+\sum_{j_1\neq j_2}\|J^{j_1j_2}_{c_1c_2}(t)\|_{L^1}
\eeaa
Continuing in this way we derive,
\beaa
\|J_n(t)\|_{L^1}&\les&N(g) n^{\frac 32} \De_0^nC^n\big(1+\frac {(n-1)}{1!}+
\frac{(n-1)(n-2)}{2!}\ldots+\frac{(n-1)\ldots (n-m)}{(m-1)!}+ \ldots
1\big )\\
&\les&n^{\frac 32}\De_0^nC^n(1+1)^{n-1} N(g)
\les n^{\frac 32} \De_0^n(2C)^n N(g),
\eeaa
as claimed in \eqref{eq:n}.
 
\end{document}